\documentclass{amsart}
\usepackage{amsmath, amssymb, latexsym,  amsthm, amscd, xy, color}
\usepackage{enumitem}

\usepackage[english]{babel}
\theoremstyle{plain}
\newtheorem{theorem}{Theorem}[section]
\newtheorem{proposition}[theorem]{Proposition}

\newtheorem{lemma}[theorem]{Lemma}
\theoremstyle{definition}

\newtheorem*{Acknowledgement}{Acknowledgements}

\newtheorem{remark}[theorem]{Remark}

\newtheorem*{alexander-hirschowitz}{\hypertarget{alex-hir}{Alexander-Hirschowitz Theorem}}

\newtheorem*{thm*}{Theorem}

\newcommand{ \rrk }{{\mathrm{rk}_{\mathbb{R}}}}

\newcommand{ \crk }{{\mathrm{rk}_{\mathbb{C}}}}

\begin{document}

\title{Real and complex Waring rank of reducible cubic forms}

\author[E. Carlini]{Enrico Carlini}
\address[E. Carlini]{Department of Mathematical Sciences, Politecnico di Torino, Turin, Italy}
\email{enrico.carlini@polito.it}

\author[C. Guo]{Cheng Guo}
\address[C. Guo]{Centre for Quantum Computation and Intelligent Systems (QCIS), Faculty of Engineering and Information Technology, University of Technology, Sydney (UTS)}
\email{Cheng.Guo@student.uts.edu.au}

\author[E. Ventura]{Emanuele Ventura}
\address[E. Ventura]{Department of Mathematics and Systems Analysis, Aalto University}
\email{emanuele.ventura@aalto.fi}

\begin{abstract}

In this paper, we study the real and the complex Waring rank of reducible cubic forms. In particular, we compute the complex rank of all reducible cubics. In the real case, for all reducible cubics, we either compute or bound the real rank depending on the signature of the degree two factor.

\end{abstract}

\maketitle


\section{Introduction}

Let $\mathbb K$ be a field. Let $F\in\mathbb{K}[x_0,\ldots,x_n]$ be a homogeneous polynomial of degree $d$. The \textit{Waring problem} for $F$ over $\mathbb K$ asks for the least value $s$ such that there exist linear forms $L_1, \ldots, L_s$ over $\mathbb K$, for which $F$ can be written as a sum of powers
\[F=\lambda_1L_1^d+\ldots+\lambda_sL_s^d,\]
where $\lambda_i\in \mathbb K$ for all $i$. Such a value $s$ is called the \textit{Waring rank} over $\mathbb K$, or the $\mathbb{K}$--\textit{rank}, of the form $F$, and it is denoted by $\mbox{rk}_\mathbb{K}(F)$. Note that the rank could be infinite for fields of positive characteristic. Moreover, when $\mathbb K=\mathbb C$ we may assume that $\lambda_i=1$ for al $i$. While, when  $\mathbb{K}=\mathbb R$, we may assume $\lambda_i=\pm 1$ for all $i$.

The notion of Waring rank, and its generalization to the case of tensors, are intensively studied also because of their many applications which include, but are not limited to, Algebraic Complexity Theory \cite{BI}, Signal Processing \cite{GCMT}, and Quantum Information Theory \cite{YCGD,YGD}.
Most applications are concerned with the real and complex cases, that is the cases in which $\mathbb{K}=\mathbb{R}$ or $\mathbb{K}=\mathbb{C}$. We will call $\rrk(F)$, respectively $\crk(F)$, the real, respectively the complex, rank of $F$.

Our knowledge of the Waring rank is very limited even for $\mathbb{K}=\mathbb{R}$ or $\mathbb{K}=\mathbb{C}$. The complex Waring rank, for example, is known for all monomials, see \cite{CCG}, but the real Waring rank is only known in the case of monomials in two variables, see \cite{BCG11, R2013}.

Since the Waring rank of a quadratic form is the rank of its associated matrix, it is natural to consider cubic forms as the next case of interest. However, the degree three case is already beyond our reach, and a complete description of the complex rank is only given when at most three variables are involved. In \cite{LT}, the three variable case is treated using projective changes of coordinates in order to obtain canonical forms of which the complex rank is then computed. In the same paper, a similar idea is used to find the complex rank of some reducible cubic forms in any number of variables.

The structure of the present paper is as follows. In Section \ref{lowerboundsection}, we introduce Theorem \ref{criterion} which is our basic tool to give lower bounds for the rank. In Section \ref{complexsection}, we recover and complete the description of the complex rank of reducible cubic forms given in \cite{LT}, see Theorem \ref{Classification of reducible cubics}. In Section \ref{realsection}, we classify real reducible cubics and we give lower and upper bounds for the real rank in Theorem \ref{realthm}.

\section{Basic facts}\label{basicfactssection}

In the next sections we study the real and the complex rank of reducible cubic forms using the action of the group $\textnormal{GL}(n+1,\mathbb{K})$ for $\mathbb{K}=\mathbb{R},\mathbb{C}$. In particular, we say that forms $F,G\in\mathbb{K}[x_0,\ldots,x_n]$ are {\bf equivalent over }$\mathbb{K}$ if
\[G(x_0,\ldots,x_n)=\lambda F((x_0,\ldots,x_n)A)\]
for some $A\in \textnormal{GL}(n+1,\mathbb{K})$ and for some $\lambda\in\mathbb{K}$ with $\lambda\neq 0$. If the field $\mathbb{K}$ is clear from the context, we will simply say that $F$ and $G$ are {\bf equivalent}.

The crucial remark is that forms equivalent over $\mathbb{K}$ have the same rank over $\mathbb{K}$.

\begin{remark}\label{orthogonalremark}
We use the orthogonal subgroups $\textnormal{O}(n+1,\mathbb{K})\subset \textnormal{GL}(n+1,\mathbb{K})$. Recall that for $\mathbb{K}=\mathbb{R}$ the action of the orthogonal subgroup is transitive, in particular all non-zero linear forms are equivalent up to an element in $\textnormal{O}(n+1,\mathbb{R})$. However, when $\mathbb{K}=\mathbb{C}$, this is no longer true and there are two non-zero equivalence classes. Namely, the linear forms $a_0x_0+\ldots+a_nx_n$ with $\sum_i a_i^2\neq 0$ are equivalent to each other, while the non-zero linear forms with $\sum_i a_i^2=0$ form a disjoint equivalence class.
\end{remark}

Finally, we recall the notion given in \cite{C2004} of a form $F$ {\bf essentially involving } $n+1$ variable. We say that $F$ essentially involves $n+1$ variables if $F$ is not equivalent to a form $G$ in fewer variables.

\section{A lower bound for the rank}\label{lowerboundsection}

In this section we work over a field $\mathbb K$ of characteristic zero. First we have the following lemma.

\begin{lemma}\label{lower bound lemma}

Let $F\in\mathbb{K}[x_0,\ldots,x_n]$ be a form of degree $d$ and set $F_k=\partial F/\partial x_k$ for $0\leq k \leq n$.  If $r$ is the minimal non-negative integer for which there exist linear forms $L_1, \cdots, L_r$, such that
\[ F_k \in <L_1^{d-1}, \cdots, L_s^{d-1}>\]
for $0 \leq k \leq n$, then $\mbox{rk}_{\mathbb K}(F) \geq r$.

\begin{proof}

Let $t=\mbox{rk}_{\mathbb K}(F)$, $F= \sum \limits_{k=1}^{t} L_k^d$ thus
$$F_k= \sum \limits_{k=1}^{t} \partial L_k^d/ \partial x_k
\\= (d-1)\sum \limits_{k=1}^{t} (\partial L_k / \partial x_k) L_k^{d-1}.$$
Hence, for each $k$, $F_k \in <L_1^{d-1}, \cdots, L_{t}^{d-1}>$. By definition of $r$, we have $\mbox{rk}_{\mathbb K}(F) \geq r$.

\end{proof}

\end{lemma}

The following result is inspired by the property of quantum states transformation via local operation and classical communication (LOCC) \cite{CDS} \cite{YCGD} \cite{YGD}.

\begin{theorem}\label{criterion}

Let $1\leq p \leq n$ be an integer, $F\in\mathbb{K}[x_0,\ldots,x_n]$ be a form, and set  $F_k=\partial F/\partial x_k$ for $0\leq k \leq n$.
If
\[
\textnormal{rk}_{\mathbb K}(F_0 + \sum_{k=1}^{p} \lambda_k F_k)\geq m
\]
for all $\lambda_k \in \mathbb{K}$
and if the forms $F_1, F_2, \cdots, F_p$ are linearly independent, then \[\textnormal{rk}_{\mathbb K}(F) \geq m+p.\]

\end{theorem}

\begin{proof}

Suppose by contradiction that $\mbox{rk}_{\mathbb K}(F) < m+p$. By Lemma \ref{lower bound lemma}, there exist $m+p-1$ linear forms $L_1, L_2, \cdots, L_{m+p-1}$, such that for each $k$, $F_k \in <L_1^{d-1}, \cdots, L_{m+p-1}^{d-1}>$.
Hence, there is a $p\times (m+p-1)$ matrix $M$ of rank $p$, such that
$$
\begin{pmatrix}
F_1\\
F_2\\
\vdots\\
F_p
\end{pmatrix}
= M
\begin{pmatrix}
L_1^{d-1}\\
L_1^{d-1}\\
\vdots\\
L_{m+p-1}^{d-1}\\
\end{pmatrix}.
$$

\noindent Performing Gaussian elimination on $M$, we can decompose $M$ as $M=M_0 M_1$, where $M_0$ is a full rank matrix, which is product of elementary matrices, and $M_1$ has the following form

$$
M_1=
\begin{pmatrix}
1 & \cdots & \cdots & \cdots   & \cdots  & \cdots  \\
0 &   1 &   \cdots  & \cdots   & \cdots  & \cdots \\
0 &   0 &    0      &    0     & 1       & \cdots  \\
\vdots  & \vdots    & \vdots   & \vdots  & \ddots & \vdots \\
0 &   0 &    0      & \cdots   &   0     &      1 & \cdots
\end{pmatrix}.
$$

\noindent Each first non-zero element in each row of $M_1$ is $1$ and these $1$'s are in different columns. Suppose the $1$ in the $k-$th row is in the $h(k)-$th column. We have $1 = h(1) < h(2) < \cdots < h(p) \leq m+p-1$. Let $H=\{ h(1),h(2),\cdots, h(p) \}$. Thus for $1 \leq k \leq p$, there exist $\mu_{k,t}$, with $h(k) < t \leq p$, such that
$L_{h(k)}^{d-1} + \sum \limits_{t=h(k)+1}^{m+p-1} \mu_{k,t} L_{j}^{d-1}
\in <F_1, F_2, \cdots, F_p>$.\\
Suppose $F_0= \sum \limits_{j=1}^{m+p-1} \nu_j L_{j}^{d-1}$. Then we have $F_0= \sum \limits_{j \in H} \nu_j L_{j}^{d-1} + \sum \limits_{j \notin  H} \nu_j L_{j}^{d-1}$. For $j \notin H$, there exists $\bar{\nu_j}$, such that
$\sum \limits_{j \in H} \nu_j L_{j}^{d-1} + \sum \limits_{j \notin  H} \bar{\nu_j} L_{j}^{d-1} \in <F_1,F_2, \cdots, F_p>$.
Therefore
$$F_0= \sum \limits_{j \in H} \nu_j L_{j}^{d-1} + \sum \limits_{j \notin  H} \nu_j L_{j}^{d-1}
= (\sum \limits_{j \in H} \nu_j L_{j}^{d-1} + \sum \limits_{j \notin  H} \bar{\nu_j} L_{j}^{d-1})+
\sum \limits_{j \notin  H} (\nu_j-\bar{\nu_j}) L_{j}^{d-1}.$$
\noindent Then
$$\sum \limits_{j \notin  H} (\nu_j-\bar{\nu_j}) L_{j}^{d-1} = F_0 -( \sum \limits_{j \in H} \nu_j L_{j}^{d-1} + \sum \limits_{j \notin  H} \bar{\nu_j} L_{j}^{d-1} ).$$

\noindent The rank of left-hand side is at most $m+p-1-|H|=m-1$, while the rank of right-hand side is at least $m$ by assumption and this gives a contradiction.

\end{proof}

Theorem \ref{criterion} is particularly effective when dealing with cubic forms and we will use it in Sections \ref{complexsection} and \ref{realsection} to find our lower bounds. In fact, since the partial derivatives of a cubic are quadratic forms, the determination of their rank is equivalent to the computation of the rank of a matrix. However, Theorem \ref{criterion} is also useful when the degree is larger than three, as shown in Proposition \ref{generalization of C}.

\section{Complex rank of reducible cubic forms}\label{complexsection}

In this section we determine the complex rank of all reducible cubic forms in $\mathbb{C}[x_0,\ldots x_n]$ hence completing the result of \cite{LT}. Since the complex rank of monomials is known by \cite{CCG}, we will only consider forms which are not equivalent to monomials.

Since equivalent forms have the same complex rank, we begin with the following classification result.

\begin{lemma}\label{complexclassification}
Let $F\in \mathbb{C}[x_0,\ldots x_n]$ be a form essentially involving $n+1$ variables which is not equivalent to a monomial. If $F$ is a reducible cubic form, then $F$ is equivalent to one and only one of the following forms:

\begin{enumerate}[label=(\roman*)]

\item\label{CubicA}

\[A=x_0(x_0^2+x_1^2+\ldots +x_n^2),\]

and in this case we say that $F$ is of type {\bf A};

\item\label{CubicB}

\[ B=x_0(x_1^2+x_2^2+\ldots +x_n^2),\]

and in this case we say that $F$ is of type {\bf B};

\item\label{CubicC}

\[C=x_0(x_0x_1+x_2^2+x_3^2+x_4^2+\ldots +x_n^2),\]

and in this case we say that $F$ is of type {\bf C}.

\end{enumerate}
\end{lemma}
\begin{proof}
Clearly, $F$ is equivalent to $L(\epsilon x_0^2 +\sum_{i=1}^n x_i^2)$ where $\epsilon\in\lbrace0,1\rbrace$. If $\epsilon=0$, then $F$ is of type {\bf B}. If $\epsilon=1$, then the type of $F$ depends on the orbit of $L=\sum_{i=0}^n a_ix_i$ under the action of the orthogonal group $\textnormal{O}(n+1,\mathbb C)$ which stabilizes the quadratic factor $\epsilon x_0^2 +\sum_{i=1}^n x_i^2$. As described in Remark \ref{orthogonalremark}, if $\sum_{i=0}^n a_i^2\neq 0$, then $F$ is of type {\bf A}. If $\sum_{i=0}^n a_i^2= 0$, $F$ is equivalent to
\[(x_0+ix_1)(x_0^2+x_1^2+\ldots+x_n^2),\]
and hence $F$ is equivalent to $C$.
\end{proof}

\begin{remark}
We can describe Lemma \ref{complexclassification} in geometric terms. Let $F=LQ$ be a cubic essentially involving $n+1$ variables which is not equivalent to a monomial. If $F$ is of type {\bf A}, then the hypersurface $V(Q)$ is not a cone and the hyperplane $V(L)$ is not tangent to the quadric. If $F$ is of type {\bf B}, then $V(Q)$ is a cone and $V(L)$ does not pass through any vertex of the quadric $V(Q)$; note that if the linear space contains any vertex of the quadric, then we can project from a vertex thus reducing to a case in fewer variables. Hence, we end up again with a reducible cubic either of type {\bf A}, {\bf B} or {\bf C} in less than $n+1$ variables.
If $F$ is of type {\bf C}, then $V(Q)$ is not a cone and $V(L)$ is tangent to the quadric; note that the condition $\sum_{i=0}^n a_i^2=0$ in Remark \ref{orthogonalremark} is equivalent to the tangency condition for the hyperplane $V(\sum_{i=0}^n a_ix_i)$ by the polar properties of quadric $V(\sum_{i=0}^n x_i^2)$.
\end{remark}

We give a complete description of the complex rank of reducible cubic forms. We prove two propositions giving an upper bound and a lower bound on the complex rank. Note that the complex rank for cubics of type {\bf A} and {\bf B} is also given in \cite{LT}, but we produce here independent proofs. Our result on cubics of type {\bf C} is, at the best of our knowledge, new. Note that, B. Segre proved that the cubic surface $V(C)\subset\mathbb P^3$ has rank $7$ and that this is the maximal rank among cubic surfaces \cite{Segre}.

\begin{proposition}\label{Upperbounds}

The cubic forms of types ${\bf A}$, ${\bf B}$, and ${\bf C}$ have complex rank at most $2n$, $2n$, and $2n+1$ respectively.

\begin{proof}

Let $A^{(k)}, B^{(k)}$ and $C^{(k)}$ denote cubic forms of types ${\bf A},{\bf B}$ and ${\bf C}$ in $k+1$ variables. \\
\indent
Let us consider the cubic forms of type ${\bf A}$. For $n=1$, we have
\begin{equation}\label{x0cubopiux0x1square}
A^{(1)}=x_0(x_0^2+x_1^2)={\frac{1}{6\sqrt{3}}}\left [ (\sqrt{3}x_0-x_1)^3+(\sqrt{3}x_0+x_1)^3\right ],
\end{equation}
whose complex rank is $2$. Suppose the statement holds for $k\leq n-1$. Let us consider the form $A^{(n)}=x_0(x_0^2+x_1^2+\ldots +x_n^2)$. If we set $A'=1/6[(x_0+x_1)^3+(x_0-x_1)^3]=1/3x_0^3+x_0x_1^2$ and $A''=x_0(2/3x_0^2+x_2^2+\ldots +x_n^2)$, then $A^{(n)}=A'+A''$. The form $A''$ is equivalent to a cubic form of type ${\bf A}$, and by induction it has complex rank at most $2(n-1)$. Hence we have that \[\crk(A^{(n)})\leq \crk(A')+\crk(A'')\leq 2+2(n-1).\]

Let us consider cubic forms of type ${\bf B}$. Suppose the statement holds for $3\leq k\leq n-1$. Let us consider the form $B^{(n)}=x_0(x_1^2+x_2^2+\ldots +x_n^2)$ and note that $B^{(2)}=x_0(x_1^2+x_2^2)$ has complex rank $4$. If we set
$B'=x_0(x_3^2+\ldots+x_n^2)$, then $B^{(n)}=B^{(2)}+B'$. The form $B'$ is equivalent to a cubic form of type ${\bf B}$ and hence, by induction, it has complex rank at most $2(n-2)$. Hence we have that \[\crk(B^{(n)})\leq 4+2(n-2)=2n.\]

Let us consider cubic forms of type ${\bf C}$ and let $C^{(n)}=x_0(x_0x_1+x_2^2+x_3^2+\ldots+x_n^2)$. Thus
$C^{(n)}=x_0^2x_1+B^{(n-1)}$, where $B^{(n-1)}$ is a form of type ${\bf B}$ in $n$ variables. Hence we have
$\crk(C^{(n)})\leq 3+2(n-1)=2n+1$.

\end{proof}

\end{proposition}

\begin{proposition}\label{Lowerbounds}

The cubic forms of types ${\bf A}, {\bf B}$, and ${\bf C}$ have complex rank at least $2n$, $2n$, and $2n+1$ respectively.

\begin{proof}

Let us consider the cubic form $A$ and Let $A_k$ denote $\partial A/\partial x_k$. The forms $A_1,\ldots,A_n$ are linearly independent and for any $\lambda_k \in \mathbb{C}$, with $1 \leq k \leq n$, we have $\crk(A_0 + \sum_{k=1}^{n} \lambda_k A_k)\geq n$. The last statement follows considering the associated matrix to the quadratic form $A_0 + \sum_{k=1}^{n} \lambda_k A_k$ and noticing that its rank is at least $n$. By Theorem \ref{criterion} with $m=n$ and $p=n$, we have $\crk(A) \geq m+p =2n$. In complete analogy we have that $\crk(B) \geq m+p =2n$. \\
\indent We now consider the cubic form (\ref{CubicC}) of type ${\bf C}$ and let $C_k$ denote $\partial C/\partial x_k$ . The forms $C_1, C_2, \ldots, C_n$ are linearly independent and for any $\lambda_k \in \mathbb{C}$, with $1 \leq k \leq n$, we have $\crk(C_0 + \sum_{k=1}^{n} \lambda_k C_k)=n+1$. The last statement is equivalent to the fact that the following matrix has non-zero determinant

$$
M=\begin{pmatrix}
\lambda_1 & 1 & \lambda_2 & \lambda_3& \lambda_4 & \lambda_5 & \cdots & \lambda_n \\
1       &   0 &    0      &    0      &    0      &    0     & \cdots & 0 \\
\lambda_2 &   0 &    1      &    0      &    0      &    0     & \cdots & 0 \\
\lambda_3 &   0 &    0      &    1      &    0      &    0     & \cdots & 0 \\
\lambda_4 &   0 &    0      &    0      &    1      &    0     & \cdots & 0 \\
\lambda_5 &   0 &    0      &    0      &    0      &    1     & \cdots & 0 \\
\vdots    & \vdots & \vdots & \vdots    &  \vdots   &   \vdots & \ddots & \vdots \\
\lambda_n &   0 &    0      &    0      &    0      &    0     & \cdots & 1
\end{pmatrix}
$$
A direct computation shows that $M$ has rank $n+1$, thus $\crk(C_0 + \sum_{k=1}^{n} \lambda_k C_k)=n+1$ for all $\lambda_k\in \mathbb C$. Hence, by Theorem \ref{criterion}, with $m=n+1$ and $p=n$ we have $\crk(C) \geq m+p =2n+1$. This concludes the proof.

\end{proof}

\end{proposition}

\begin{theorem}\label{Classification of reducible cubics}
Let $F\in \mathbb{C}[x_0,\ldots x_n]$ be a form essentially involving $n+1$ variables which is not equivalent to a monomial. If $F$ is a reducible cubic form, then one and only one of the following holds:
\begin{itemize}
\item $F$ is equivalent to

\[x_0(x_0^2+x_1^2+\ldots +x_n^2),\]
and $\crk F=2n$.

\item $F$ is equivalent to

\[x_0(x_1^2+x_2^2+\ldots +x_n^2),\]
and $\crk F=2n$.

\item $F$ equivalent to

\[x_0(x_0x_1+x_2x_3+x_4^2+\ldots +x_n^2),\]
and $\crk F=2n+1$.
\end{itemize}
\end{theorem}

\begin{proof}
It is enough to combine Lemma \ref{complexclassification} and Propositions \ref{Upperbounds}, and \ref{Lowerbounds}.
\end{proof}

We conclude with an application of the result to forms of type {\bf C}.

\begin{proposition}\label{generalization of C}

If $F=x_0^{d-1} x_1 + x_0^{d-2} \sum_{k=2}^n x_k^2$, then $\crk(F)=(d-1)n+1$.

\end{proposition}

\begin{proof}

The proof is by induction on the degree of $F$. Let $F^{(r)}=x_0^{r-1} x_1 + x_0^{r-2} \sum_{k=2}^n x_k^2$. Thus, $F^{(3)}$ is a reducible cubic form of type ${\bf C}$ and the statement holds for $r=3$. Assume the statement holds for every $r\leq d-1$ and we prove it for degree $d$. Let $t=d-1$ and $F^{(d)}=x_0^{d-1} x_1 + x_0^{d-2} \sum_{k=2}^n x_k^2 =x_0^{t} x_1 + x_0^{t-1} \sum_{k=2}^n x_k^2$. \\
We have the partial derivatives $F^{(d)}_0= t x_0^{t-1} x_1 + (t-1) x_0^{t-2} \sum_{k=2}^n x_k^2,
F^{(d)}_1 = x_0^t$, and $F^{(d)}_k=2 x_0^{t-1} x_k$ for $k \geq 2$. The forms $F^{(d)}_k$ for $1\leq k\leq n$ are linearly independent. Consider the form $F^{(d)}_0+ \sum \limits_{k=1}^n \mu_k F^{(d)}_k$ for any $\mu_k\in \mathbb C$. We have
$$
F^{(d)}_0+ \sum_{k=1}^n \mu_k F^{(d)}_k= t x_0^{t-1} x_1 + (t-1) x_0^{t-2} \sum_{k=2}^n x_k^2 + \mu_1 x_0^{t}+
\sum_{k=2}^n \mu_k 2 x_0^{t-1} x_k
$$
$$
= x_0^{t-1} (\mu_1 x_0 + t x_1)  + x_0^{t-2} \sum_{k=2}^n( (t-1) x_k^2 + 2 \mu_k x_0 x_k)
$$
$$
= x_0^{t-1} y_1 + x_0^{t-2} \sum_{k=2}^n y_k^2,
$$
\noindent where $y_1=\mu_1 x_0 + t x_1$ and $y_k= \sqrt{t-1} (x_k + 2 \mu_k/\sqrt{t-1} x_0)$. By induction and by Theorem \ref{criterion}, we have $\crk(F^{(d)})\geq (d-1)n+1$ for all $d\geq 3$. On the other hand, $\crk(F^{(d)})\leq (d-1)n+1$, since this is the bound given by the monomials in $F$.

\end{proof}

\section{Real rank of reducible cubic forms}\label{realsection}

In this section we study the real rank of reducible cubic forms in $\mathbb{R}[x_0,\ldots x_n]$. Since the rank is invariant under the action of $\textnormal{GL}(n+1,\mathbb{R})$ we only need to study the rank of non-equivalent forms. The real rank of all of cubics in three variables are determined in \cite{Ban}. \\


We will use the following classification result.

\begin{lemma}\label{realclassification} If $F\in \mathbb{R}[x_0,\ldots x_n]$ is a reducible cubic form essentially involving $n+1$ variables, then $F$ is equivalent to one and only one of the following:
\begin{enumerate}[label=(\roman*)]
\item\label{onereal} \[x_0(\sum_{i=1}^n \epsilon_i x_i^2),\]
where $\epsilon_i\in\lbrace -1,+1\rbrace$ for $1\leq i\leq n$;
\item\label{tworeal} \[x_0(\sum_{i=0}^n \epsilon_i x_i^2),\]
where $\epsilon_i\in\lbrace -1,+1\rbrace$ for $0\leq i\leq n$;
\item\label{threereal} \[(\alpha x_0+x_p)(\sum_{i=0}^n \epsilon_i x_i^2),\]
for $\alpha\neq 0$, where $\epsilon_0=\ldots=\epsilon_{p-1}=1$ and $\epsilon_p=\ldots=\epsilon_n=-1$ for $1\leq p\leq n$.
\end{enumerate}
\end{lemma}

\begin{proof}
Let $F=LQ$ where $L$ ia a linear form and $Q$ is a quadratic form. Clearly $F$ is equivalent to $L'\sum_{i=0}^n \epsilon_i x_i^2$ where $\epsilon_0\in\lbrace -1,0,1\rbrace$ and $\epsilon_1,\ldots,\epsilon_n\in\lbrace -1,1\rbrace$. Note that at most one of the coefficients $\epsilon_i$ could be zero, otherwise $F$ would essentially involve less than $n+1$ variables.

If $\epsilon_0=0$, then $L'$ is a linear form linearly independent with $x_1,\ldots,x_n$, otherwise $F$ would essentially involve less than $n+1$ variables. Thus, $F$ is equivalent to
$$x_0(\sum_{i=1}^n \epsilon_i x_i^2),$$
and hence $F$ is equivalent to the form in \ref{onereal}.

If all the coefficients $\epsilon_i$ are non zero and have the same sign, then $F$ is equivalent to $L'(\sum_{i=0}^n x_i^2)$. Note that the quadratic part is stabilized by the orthogonal group $\textnormal{O}(n+1,\mathrm{R})\subset \textnormal{GL}(n+1,\mathrm{R})$. Since the action of $\textnormal{O}(n+1)$ is transitive, $F$ is equivalent to
$$x_0\sum_{i=1}^n x_i^2,$$
and hence $F$ is equivalent to the form in \ref{tworeal}.

If all the coefficients $\epsilon_i$ are non zero and do not have the same sign, then $F$ is equivalent to
$$L'(x_0^2+\ldots+x_{p-1}^2-x_p^2\ldots-x_{n}^2)$$
for some $p$, $1\leq p\leq n$. Note that the quadratic part is stabilized by $\textnormal{O}(p,\mathbb R)\times \textnormal{O}(n+1-p,\mathbb R)$. Since the actions of $\textnormal{O}(p,\mathbb R)$ and $\textnormal{O}(n+1-p,\mathbb R)$ are transitive, $F$ is equivalent to
$$(\alpha x_0+\beta x_p)(\sum_{i=0}^n \epsilon_i x_i^2),$$
for $\alpha,\beta\in\mathbb{R}$. Hence, if $\alpha=0$ or $\beta=0$, then $F$ is equivalent to the form in \ref{tworeal}. If $\alpha\neq 0$ and $\beta\neq 0$, then $F$ is equivalent to the form in \ref{threereal}. This concludes the proof.

\end{proof}

We now study the real ranks of the forms given in Lemma \ref{realclassification}. We begin with case \ref{onereal} of the lemma.

\begin{proposition}\label{Type 1}
If $F=x_0(\sum_{i=1}^n \epsilon_i x_i^2)$ where $\epsilon_i\in\lbrace -1,+1\rbrace$, then $2n\leq\rrk(F)\leq 2n+1$. Moreover, if $\sum_{i=1}^n\epsilon_i=0$, then $\rrk(F)=2n$.
\end{proposition}

\begin{proof}
Note that
\begin{equation}\label{x0xisquareequation}
x_0x_i^2=1/6[(x_0+x_i)^3+(x_0-x_i)^3]-1/3x_0^3.
\end{equation}
Since $F=\sum_{i=1}^n \epsilon_i x_0x_i^2$, we can find a real sum of powers decomposition of $F$ involving at most $2n+1$ cubes, and thus $\rrk(F)\leq 2n+1$. By Theorem \ref{Classification of reducible cubics}, we have $2n=\crk(F)\leq\rrk(F)$ and the inequality is hence proved.
To prove the equality, note that the condition $\sum_{i=1}^n\epsilon_i=0$ yields that the coefficient of the monomial $x_0^3$ is zero in the decomposition of $F$ given by equation \eqref{x0xisquareequation}. Hence we have $\rrk(F)\leq 2n$  and this concludes the proof.

\end{proof}

We now consider case \ref{tworeal} of Lemma \ref{realclassification}.

\begin{proposition}\label{Type 2}
If $F=x_0(\sum_{i=0}^n \epsilon_i x_i^2)$ where $\epsilon_i\in\lbrace -1,+1\rbrace$, then $2n\leq\rrk(F)\leq 2n+1$. If $\epsilon_0=\ldots=\epsilon_n$, then $\rrk(F)=2n$. If $\epsilon_0\neq \epsilon_1$ and $\epsilon_1=\ldots=\epsilon_n$, then $\rrk(F)=2n+1$.
\end{proposition}

\begin{proof}
Using equation \eqref{x0xisquareequation} we get that $\rrk(F)\leq 2n+1$. Since $2n=\crk(F)\leq\rrk(F)$ the inequality is proved.
Suppose that $\epsilon_0=\ldots=\epsilon_n$; we proceed by induction on $n$. Note that it is enough to consider $\epsilon_0=1$. If $n=1$, then $F=x_0(x_0^2+x_1^2)$ and, by equation \eqref{x0cubopiux0x1square} we have that $\rrk(F)=2$. Now suppose that, for every $2\leq k\leq n-1$, $F^{(k)}=x_0(x_0^2+\ldots+x_k^2)$ has real rank $2k$. Note that $F^{(k)}=A'+A''$, where
$$A'=x_0(1/2x_0^2+x_1^2)$$
and
$$A''=x_0(1/2x_0^2+x_2^2+\ldots+x_n^2).$$
Note that $A''$ is equivalent to $F^{(n-1)}$. Thus, $\rrk(A'')=2(n-1)$ by induction. By the $n=1$ case we have that $\rrk(A')=2$, and hence we conclude that $\rrk(F^{(n)})\leq 2n$. Recalling that $2n\leq\crk(F)$, this concludes the proof.

Suppose that $\epsilon_0\neq\epsilon_1$ and $\epsilon_1=\ldots=\epsilon_n$; it is enough to consider $\epsilon_0=1$. Let $M$ be the matrix associated to the quadratic form
$$
F_0+\sum_{i=1}^n \lambda_i F_i,
$$
thus
$$
M=\begin{pmatrix}
    3     & -\lambda_1  & -\lambda_2 & \dots & \ldots & -\lambda_n \\
    -\lambda_1 &  -1 & 0 & \dots & 0  & 0 \\
    \vdots & \vdots & -1 & \ldots & 0 & 0 \\
     -\lambda_p     & 0 & \ldots & -1 & 0 &\vdots \\
    \vdots &  0 & 0 & \ldots & \ddots & 0 \\
    -\lambda_n & 0 & \ldots & \ldots & \ldots & -1
\end{pmatrix}.
$$

\noindent The determinant of $M$ is $(-1)^{n}(\sum_{i=1}^n \lambda_i^2+3)$ which is always non zero for real values of the parameters $\lambda_i$. By Theorem \ref{criterion}, we conclude that $2n+1\leq \rrk(F)$, and hence $\rrk(F)=2n+1$. This concludes the proof.

\end{proof}

\begin{remark}

Note that we know the real rank of $x_0(x_0^2+x_1^2+\ldots +x_n^2)$, but we do not know in general the real rank of $x_0(x_1^2+\ldots +x_n^2)$. To see why the latter case is more difficult, consider the case $n=2$. The form $F=x_0(x_1^2+x_2^2)$ is a monomial with distinct factors and thus $\crk(F)=4$ by \cite{CCG}. However, $F$ is not a monomial over the reals. We can prove that $\rrk(F)=5$ using apolarity as follows, see \cite{Ban} for an alternative proof. Assume by contradiction that $\rrk(F)=4$. Since the complex rank of $F$ is $4$, by Proposition 21 in \cite{BBT}, we have that there is a set of four points apolar to $F$ which is the complete intersection of two conics, and thus the points lie on three reducible conics. If we assume that $\rrk(F)=4$ these points have real coordinates, and thus the reducible conics are of the form $LM$ for real linear forms $L$ and $M$. However, a direct computation shows that the ideal $F^\perp$ does not contain degree two polynomial of such a form hence a contradiction.
\end{remark}

We now consider case \ref{threereal} of Lemma \ref{realclassification}.

\begin{proposition}\label{Type 3}

If $(\alpha x_0+x_p)(\sum_{i=0}^n \epsilon_i x_i^2)$, where $\alpha\neq 0$, $\epsilon_0=\ldots=\epsilon_{p-1}=1$, and $\epsilon_p=\ldots=\epsilon_n=-1$ for $2\leq p\leq n$, then $2n\leq\rrk(F)\leq 2n+3$. If $\alpha=-1$ or $\alpha=1$, then $2n+1\leq\rrk(F)\leq 2n+3$.

\end{proposition}

\begin{proof}
We have
$$
F=(\alpha x_0+x_p)(\epsilon_0x_0^2+\epsilon_px_p^2)+(\alpha x_0+x_p)(\sum_{i=1}^n \epsilon_i x_i^2),
$$
where the first summand can be written as the sum of four cubes of linear forms using \eqref{x0cubopiux0x1square}. Analogously, the second summand can be written as a sum of $n-1$ polynomials of real rank two and the cubic form $(\alpha x_0+x_p)^3$. Hence, we have $\rrk(F)\leq 2(n-1)+5=2n+3$. The complex rank of $F$ gives the required lower bound and so we get $2n\leq\rrk(F)\leq 2n+3$.

If $\alpha=-1$ or $\alpha=+1$, then $F$ is of type {\bf C} and hence $2n+1\leq\rrk(F)\leq 2n+3$. This concludes the proof.

\end{proof}

In conclusion we have the following result.

\begin{theorem}\label{realthm}
If $F\in \mathbb{R}[x_0,\ldots x_n]$ is a reducible cubic form essentially involving $n+1$ variables, then one and only one of the following holds:
\begin{itemize}

\item $F$ is equivalent to $x_0(\sum_{i=1}^n \epsilon_i x_i^2)$, where $\epsilon_i\in\lbrace -1,+1\rbrace$ for $1\leq i\leq n$ and
\[2n\leq\rrk(F)\leq 2n+1.\]
Moreover, if $\sum_{i=1}^n\epsilon_i=0$, then $\rrk(F)=2n$.

\item $F$ is equivalent to $x_0(\sum_{i=0}^n \epsilon_i x_i^2)$, where $\epsilon_i\in\lbrace -1,+1\rbrace$ for $0\leq i\leq n$ and
\[2n\leq\rrk(F)\leq 2n+1.\]
Moreover, if $\epsilon_0=\ldots=\epsilon_n$, then $\rrk(F)=2n$. If $\epsilon_0\neq \epsilon_1$ and $\epsilon_1=\ldots=\epsilon_n$, then $\rrk(F)=2n+1$.

\item $F$ is equivalent to $(\alpha x_0+x_p)(\sum_{i=0}^n \epsilon_i x_i^2)$, for $\alpha\neq 0$, where $\epsilon_0=\ldots=\epsilon_{p-1}=1$ and $\epsilon_p=\ldots=\epsilon_n=-1$ for $1\leq p\leq n$, and
\[2n\leq\rrk(F)\leq 2n+3.\]
Moreover, if $\alpha=-1$ or $\alpha=1$, then $2n+1\leq\rrk(F)\leq 2n+3$.
\end{itemize}
\end{theorem}
\begin{proof}
It is enough to consider Lemma \ref{realclassification} and to combine Propositions \ref{Type 1},\ref{Type 2}, and \ref{Type 3}.
\end{proof}

\begin{Acknowledgement}

The third author thanks Riccardo Re and Zach Teitler for helpful discussions. The first author acknowledges the financial support of Monash University, and of the Simons Institute. The first author is a members of  the GNSAGA group of INdAM. The first and third authors are supported by the framework of PRIN 2010/11 ``Geometria delle variet\`a algebriche'', cofinanced by MIUR.

\end{Acknowledgement}


\end{document}